\documentclass[12pt]{amsart}

\usepackage{amsfonts, amscd,color,amsmath,amssymb,amsthm}

\DeclareMathOperator{\inte}{int}

\newtheorem{theorem}{Theorem}[section]
\newtheorem{lemma}[theorem]{Lemma}

\newtheorem{proposition}[theorem]{Proposition}

\theoremstyle{definition}
\newtheorem{definition}[theorem]{Definition}

\theoremstyle{remark}

\numberwithin{equation}{section}

\begin{document}

\title[Hyperspaces of continua with connected boundaries]{Hyperspaces of continua with connected boundaries in $\pi$-Euclidean Peano continua}

\author{Pawe{\l}  Krupski}

\email{pawel.krupski@pwr.edu.pl}
\address{Department of Computer Science, Faculty of Fundamental Problems of Technology, Wroc\l aw University of Science and Technology, Wybrze\.{z}e Wyspia\'nskiego 27, 50-370 Wroc\l aw, Poland.}
\thanks{The author was partially supported by the Faculty of Fundamental Problems of Technology, Wroc\l aw University of Science and Technology, grant 0401/0017/17}

\date{\today}

\begin{abstract} Let $X$ be a nondegenerate Peano unicoherent continuum.  The  family $CB(X)$ of proper subcontinua of $X$ with connected boundaries is a  $G_\delta$-subset of the hyperspace $C(X)$ of all subcontinua of $X$. If every nonempty open subset of  $X$ contains an open subset homeomorphic to  $\mathbb R^n$ (such space is called $\pi$-$n$-Euclidean) and $2\le n<\infty$,  then $C(X)\setminus CB(X)$ is  recognized as an $F_\sigma$-absorber in $C(X)$; if additionally,  no one-dimensional subset separates $X$, then the family of all members of $CB(X)$ which separate $X$ is a $D_2(F_\sigma)$-absorber in $C(X)$, where $D_2(F_\sigma)$ denotes the small Borel  class of  differences of two $\sigma$-compacta.
\end{abstract}

\maketitle

All continua in the paper are metric. For a continuum $X$,
we consider the hyperspaces
$$2^{X}=\{A\subset X: \text{$A$ is closed and nonempty}\}$$ and $$C(X)= \{A\in  2^X: \text{$A$ is connected}\}$$
 with the Hausdorff metric. Define
 $$CB(X)= \{A\in C(X)\setminus \{X\}: \text{$Bd(A)$ is connected}\},$$
 where $Bd(A)$ denotes the boundary of $A$ in $X$.

\section{Evaluation of the Borel complexity  of $CB(X)$}

 K. Kuratowski observed in~\cite[p. 156]{Ku} that, for any compact nondegenerate $X$, the function $\alpha:2^X\setminus \{X\} \to 2^X$,    $\alpha(A)=\overline{X\setminus A}$ is lower semicontinuous, hence of the first Borel class, while the function $Bd: 2^X\setminus \{X\} \to 2^X$, $A\mapsto Bd(A)=A\cap \overline{X\setminus A} $, is of the second  Borel class for any nondegenerate continuum $X$.
 It means that, for each nondegenerate continuum $X$, the preimage $\alpha^{-1}(D)$ of a closed subset $D\subset 2^X$ is $G_\delta$  and the preimage $Bd^{-1}(D)$ is $F_{\sigma\delta}$ in $2^X\setminus \{X\}$. It follows that the preimages are also $G_\delta$ and  $F_{\sigma\delta}$ subsets of $2^X$.   Since $CB(X)=C(X)\cap Bd^{-1}(C(X))$, we have the following proposition.

 \begin{proposition}\label{p1}
 If $X$ is a nondegenerate continuum, then $CB(X)$  is an $F_{\sigma\delta}$-subset of $2^X$.
 \end{proposition}

Actually,  in many cases the Borel class of $CB(X)$ can be reduced. For example, one can easily see that
\begin{enumerate}
\item
if $X$ is the circle $S^1$ then  $CB(X)\overset{top}=S^1$ is   compact;
\item
 $CB(I)$    is homeomorphic to the real line $\mathbb R$, so it is an absolute $G_\delta$-set;
 \item
 if $X$ is an indecomposable nondegenerate continuum, then the family $CB(X)$ is also $G_\delta$ since it is equal to $C(X)\setminus \{X\}$.
 \end{enumerate}

\begin{lemma}\label{l1}
If $X$ is a nondegenerate  unicoherent continuum then   $CB(X)=C(X)\cap \alpha^{-1}(C(X))$.
\end{lemma}

\begin{proof}
If $C\in CB(X)$,  then $Bd(C)=C\cap \overline{X\setminus C}$ is connected and since $X=C\cup \overline{X\setminus C}$ is connected, the set $\overline{X\setminus C}$ is also connected by a well known fact~\cite[Corollary 5, p. 133]{Ku1}. Thus $C\in C(X)\cap \alpha^{-1}(C(X))$.

The converse implication follows directly from the unicoherence of $X$.
\end{proof}

Since $\alpha$ is of the first Borel class, Lemma~\ref{l1} implies the following proposition.

\begin{proposition}\label{p2}
If $X$ is a nondegenerate  unicoherent continuum then   $CB(X)$ is a $G_\delta$-subset of $C(X)$.
\end{proposition}

\section{Subcontinua with connected boundaries in $\pi$-Euclidean unicoherent Peano continua}

Recall that a family $\mathcal U$ of nonempty open subsets of a space $Y$ is a $\pi$-base of $Y$ if each nonempty open subset  of $Y$ contains some $U\in\mathcal U$.

\begin{definition}
A space $Y$ is said to be   $\pi$-Euclidean if there exist $n\in\mathbb N$ and a $\pi$-base of $Y$ whose elements are homeomorphic to $\mathbb R^n$;  such $Y$ will be also called $\pi$-$n$-Euclidean.
\end{definition}

In this section we assume that $X$ is a nondegenerate  unicoherent Peano continuum without free arcs.
Then the hyperspace $C(X)$ is a Hilbert cube~\cite{CS}.

\

One can immediately notice  that every subcontinuum of $X$ is approximated  (in the Hausdorff metric) by arcs  and the arcs  are nowhere dense in $X$, so they belong to $CB(X)$. Thus, we have
\begin{proposition}\label{p3}
$CB(X)$ is a dense $G_\delta$-subset of $C(X)$.
\end{proposition}

Imposing yet an additional structure on $X$ we can fully characterize the family $CB(X)$ in the next theorem which is our main result of this section.

Notice that if $n\ge 2$ then a  $\pi$-$n$-Euclidean space contains no free arcs.

\begin{theorem}\label{t1}
If $X$ is $\pi$-$n$-Euclidean, $n\ge 2$, then there is a homeomorphism $h:I^\infty \to C(X)$ such that $h((0,1)^\infty)=  CB(X)$.
\end{theorem}

 In proving the theorem we will rather concentrate on the complement $C(X)\setminus CB(X)$ and show that it is an $F_{\sigma}$-absorber  in  $C(X)$. For a basic terminology and facts on such absorbers the reader is referred to~\cite{M} and~\cite{Ca}.   Recall here that, given a class $\mathcal M$   of spaces which is topological (i.e., if $M\in \mathcal M $ then each homeomorphic image of $M$ is in $\mathcal M$) and closed hereditary (i.e., each closed subset of $M\in \mathcal M$ is in $\mathcal M$), a subset $A$ of a Hilbert cube $\mathcal Q$ is called an   $\mathcal M$-\emph{absorber}  in $\mathcal Q$ if
 \begin{enumerate}
\item $A\in \mathcal Q$,
\item $A$ is strongly $\mathcal M$-universal, i.e.  for each
$M\subset I^\omega$ from the class $\mathcal M$ and
each  compact set $K\subset I^\omega$, any embedding $f:I^\omega \to\mathcal Q$ such that
$f(K)$ is a $Z$-set in $\mathcal Q$
can be approximated arbitrarily closely
 by an embedding $g:I^\omega\to \mathcal Q$ such that
   $g(I^\omega)$ is a $Z$-set in $\mathcal Q$,
$g|K=f|K$ and $g^{-1}(A)\setminus K=M\setminus K$.
\item $A$ is contained in a $\sigma Z$-set in $\mathcal Q$.
\end{enumerate}
If $A\subset\mathcal Q$ is an $\mathcal M$-absorber in  $\mathcal Q$ and $B\subset I^\infty$ is an $\mathcal M$-absorber in $I^\infty$, then there is a homeomorphism $h:I^\infty\to \mathcal Q$ such that $h(B)=A$~\cite[Theorem 5.5.2]{M}.

We will be interested in two classes $\mathcal M$: the Borel class of $F_{\sigma}$-subsets of the Hilbert cube and the small Borel class $D_2(F_\sigma)$ of all subsets of the Hilbert cube that are differences of two $F_\sigma$-sets.
In the first case, property (3) above is redundant in presence of (1) and (2) (see~\cite[Theorem 5.3]{BGM}).
The pseudo-boundary $\partial(I^\infty)= \{(t_i)\in I^\infty: \text{$t_i\in\{0,1\}$ for some $i$}\}$ is a standard $F_{\sigma}$-absorber in $I^\infty$, while  $\partial(I^\infty)\times (0,1)^\infty$ is a $D_2(F_\sigma)$-absorber in  $I^\infty\times I^\infty$.

 \begin{proposition}\label{p4}
If $X$ contains an open subset homeomorphic to $\mathbb R^n$, $2\le n<\infty$, then $C(X)\setminus CB(X)$ is $F_{\sigma}$-universal, i.e., for each $F_{\sigma}$-subset $M$ of the Hilbert cube $I^\infty$,  there is an embedding $\varphi:I^\infty\to C(X)$ such that
\begin{equation}\label{varphi}
\varphi^{-1}\bigl(C(X)\setminus CB(X)\bigr)=M.
\end{equation}
\end{proposition}

\begin{proof}
For convenience,  assume that $X$ contains $(-3,3)^n$ as an open subset.
Let us construct an embedding
$\Theta:I^\infty\to C(X)$ such that
\begin{equation}
 \Theta\bigl((t_k)\bigr)\in C(X)\setminus CB(X)\quad\text{if and only if}\quad (t_k)\in \partial(I^\infty).
\end{equation}

Denote $D:=I^n\setminus (\frac13,\frac23)^n$ and, for $t\in [0,1]$,  $$D(t):=D\setminus \bigl(\frac13,\frac13(1+t)\bigr)\times \bigl[0,\frac13\bigr]\times \bigl(\frac13,\frac23\bigr)^{n-2}$$
(let us agree that $D(t)=  D\setminus \left(\frac13,\frac13(1+t)\right)\times \left[0,\frac13\right]$ if $n=2$ and $D(0)=D$),
$$D_k(t):= \left(\frac1{2^k}, 0,\dots, 0\right) + \frac1{2^k}D\bigl(t(1-t)\bigr),\quad k\in\mathbb N,$$ and consider segments
 $L_k(t)$ in  $(-3,3)^n$ from $$\left(\frac1{2^k}, \frac{-t}{2^k}, 0,\dots, 0\right) \quad\text{to}\quad \left(\frac1{2^k}, \frac{1+t}{2^k},0,\dots,0\right).$$

So, $D$ is the cube $I^n$ with the smaller open cube $D'=(\frac13,\frac23)^n$ subtracted; its boundary in  $(-3,3)^n$ has two components.   $D(t)$ is obtained from $D$ by removing a cube $D'(t)=(\frac13,\frac13(1+t))\times [0,\frac13]\times (\frac13,\frac23)^{n-2}$ adjacent to the both components; the size of  $D'(t)$ continuously depends on parameter $t$ so that the boundary of $D(t)$ is connected if and only if $t>0$.
The sets $D_k(t_k)$ are copies of $D\bigl(t_k(1-t_k)\bigr)$ scaled by
 factors $\frac1{2^k}$ and shifted by vectors $\left(\frac1{2^k}, 0,\dots, 0\right)$, correspondingly.
 The union of all $D_k(t_k)$'s compactified by the point $(0,\dots, 0)$  is a continuum whose  boundary is  connected if and only if each parameter $t_k$ is strictly between 0 and 1;  moreover, the continuum continuously depends on sequence $(t_k)$ but not in a one-to-one fashion. In order to get the one-to-one correspondence, we attach  segments $L_k(t_k)$ (Figure~\ref{fig}). Formally,
we define an embedding
\begin{equation}\label{e:theta}
\Theta\bigl((t_k)\bigr)=\{(0,\dots, 0)\} \cup \bigcup_{k=1}^\infty \bigl(L_k(t_k) \cup D_k(t_k)\bigr)\subset  (-3,3)^n.
\end{equation}

\begin{figure}[ht]
\setlength{\unitlength}{1mm}
\begin{picture}(135,40)
\thicklines

\put(80,0){\line(1,0){9}}
\put(95,0){\line(1,0){12}}
\put(80,27){\line(1,0){27}}
\put(107,0){\line(0,1){27}}
\put(80,-8){\line(0,1){43}}

\put(82,32){\makebox(0,0){\scriptsize$\}t_1$}}
\put(92,-3){\makebox(0,0){\scriptsize$\underbrace{}$}}
\put(92,-5){\makebox(0,0){\scriptsize$t_1(1-t_1)$}}

\put(89,0){\line(0,1){18}}
\put(95,0){\line(0,1){9}}
\put(89,18){\line(1,0){9}}
\put(95,9){\line(1,0){3}}
\put(98,9){\line(0,1){9}}

\put(62,18){\line(1,0){18}}
\put(62,12){\line(0,1){13}}
\put(62,0){\line(1,0){6}}
\put(71,0){\line(1,0){9}}
\put(68,0){\line(0,1){12}}
\put(68,12){\line(1,0){6}}
\put(74,6){\line(0,1){6}}
\put(71,0){\line(0,1){6}}
\put(71,6){\line(1,0){3}}
\put(62,-7){\line(0,1){7}}

\put(64,22){\makebox(0,0){\scriptsize$\}t_2$}}
\put(70,-3){\makebox(0,0){\scriptsize$\underbrace{}$}}
\put(70,-5){\makebox(0,0){\scriptsize$t_2(1-t_2)$}}

\put(50,0){\framebox(12,12)}
\put(54,4){\framebox(4,4)}

\put(42,-2){\line(0,1){12}}
\put(42,0){\line(1,0){3}}
\put(46,0){\line(1,0){4}}
\put(46,0){\line(0,1){3}}
\put(46,3){\line(1,0){1}}
\put(47,3){\line(0,1){2}}
\put(42,8){\line(1,0){8}}
\put(45,0){\line(0,1){5}}
\put(46,0){\line(1,0){4}}
\put(45,5){\line(1,0){2}}

\put(54,14){\makebox(0,0){\tiny$t_3=0$}}

\put(45,9){\makebox(0,0){\tiny$\}t_4$}}

\put(32,0){\makebox(0,0){$\cdots$}}

\put(20,0){\makebox(0,0){\textbf{$\cdot$}}}

\put(19,-2){\makebox(0,0){\scriptsize$(0,0)$}}

\end{picture}
\vspace{.5cm}
\caption{ $\Theta\bigl((t_1,t_2,0, t_4 \dots)\bigr)$ for $n=2$.}\label{fig}
\end{figure}
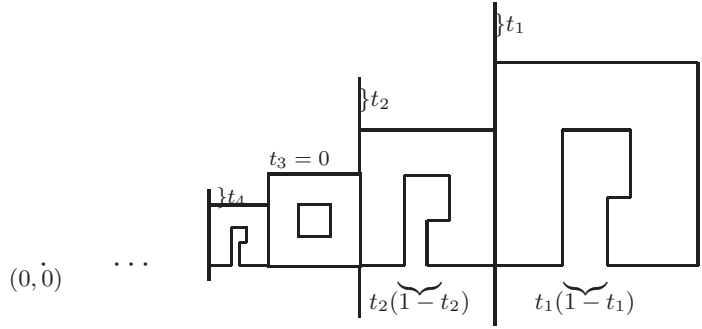

Since  the pseudo-boundary
 $\partial(I^\infty)$ is strongly $F_\sigma$-universal,  it is $F_\sigma$-universal, in particular. So,   there exists, for each $ F_\sigma$-set $M\subset I^\infty$,  an embedding $\chi: I^\infty \to I^\infty$ such that
$\chi^{-1}(\partial(I^\infty))=M$. Hence, the composition $\varphi=\Theta\chi:I^\infty \to C(X)$  satisfies~\eqref{varphi}.

\end{proof}

\newpage

 \begin{lemma}\label{l5}
If $X$ is $\pi$-$n$-Euclidean, $2\le n<\infty$, then $C(X)\setminus CB(X)$ is strongly $F_{\sigma}$-universal.
\end{lemma}

\begin{proof}
The proof is based on a  technique developed in~\cite{BC,Ca}. For our purpose, we  closely follow its rough description given in~\cite[Section 3.2]{KS}.
Without loss of generality, we can assume that
  $\varphi$ from Proposition~\ref{p4}  satisfies
  \begin{equation}\label{e:varphi}
\varphi(I^\infty)\subset C\bigl((0,1)^n\bigr)\subset C(I^n)\subset C(X).
\end{equation}

Given an $F_{\sigma}$-subset  $M$ of $I^\omega$,
  for each open non-empty subset $U$ of $X$ and an open copy of $(0,1)^n$ in $U$, let $\varphi_U:I^\omega \to C(U)$ be a composition of $\varphi$ with an embedding of $C(I^n)$ into the hyperspace of that copy. Then
  $$\varphi_U^{-1}\bigl(C(X)\setminus CB(X)\bigr)=M$$
   which means that

  \emph{$\varphi_U(q)\in C(X)\setminus CB(X)$ if and only if  $\varphi(q)\in C(X)\setminus CB(X)$ if and only if $q\in M$.}

   Notice that the equivalences remain valid if, given $q\notin K$, we add to  the one-dimensional part $A(q)$,  appearing in the construction of embedding $g(q)$, finitely many pairwise disjoint sets $\varphi_{U_i}(q)$, $i< m$ for some $m\in\mathbb N$,  such that  $A(q)\cap \varphi_{U_i}(q)$ is a singleton for each $i$.  It follows that $g^{-1}\bigl(C(X)\setminus CB(X)\bigr)\setminus K=M\setminus K$.
  Now, the construction of embedding $g$, as presented in~\cite[Section 3.2]{KS}, satisfies all the required conditions.

\end{proof}

\begin{proof}[Proof of Theorem~\ref{t1}]
By Proposition~\ref{p2} and Lemma~\ref{l1}, $C(X)\setminus CB(X)$ is an $F_\sigma$-absorber in $\mathcal Q=C(X)$. Hence, there exists a homeomorphism $h:I^\infty \to C(X)$ such that $h\bigl(\partial(I^\infty)\bigr) = C(X)\setminus CB(X)$.
\end{proof}

\subsection*{Separators with connected boundaries}

\

Now, let
$$S(X):=\{C\in 2^X: \text{$C$ separates $X$}\},\quad N(X):= \{C\in 2^X: \inte C=\emptyset\}.$$

By~\cite[Proposition 5.1]{KS} and Proposition~\ref{p2}, we get

\begin{proposition}\label{p5}
 $S(X)\cap CB(X)\in D_2(F_\sigma)$.
 \end{proposition}

 It is proved in~\cite[Theorem 5.8]{KS} that  $S(X)\cap N(X)\cap C(X)$ is a $D_2(F_\sigma)$-absorber in $C(X)$ if $X$ satisfies hypotheses of Theorem~\ref{t1}, $n\ge3$ and no subset of dimension $\le 1$ separates $X$.
In a similar way, we will show that $S(X)\cap CB(X)$ is also a $D_2(F_\sigma)$-absorber in $C(X)$ for such $X$.

\begin{proposition}\label{p6}
If $X$ contains an open subset homeomorphic to  $\mathbb R^n$, $2\le n<\infty$, then $S(X)\cap CB(X)$ is $D_2(F_\sigma)$-universal in $C(X)$, i.e., for each $D_2(F_\sigma)$-subset $M$ of  $I^\infty$,  there is an embedding $f:I^\infty\to C(X)$ such that
\begin{equation}\label{f}
f^{-1}\bigl(S(X)\cap CB(X)\bigr)=M.
\end{equation}
\end{proposition}

\begin{proof}
Without loss of generality, we can again  assume  that  $(-3,3)^n$ is an open subset of  $X$.
In~\cite[Proposition 5.2]{KS}, an embedding $\Psi:I^\infty\to C(X)$ is constructed such that
\begin{equation}\label{e:psi}
\Psi\bigl((q_k)\bigr)\in  S(X)\quad\text{if and only if}\quad (q_k)\in \partial(I^\infty).
\end{equation}

For each $(q_k)$, continuum $\Psi\bigl((q_k)\bigr)$ can be located in $[-1,0]\times [-1,1]^{n-1}$  as a subset consisting of the segment $L'=[-1,0]\times \{0,\dots,0\}$ and of a subset of the union of  cubes $[\frac{-1}{2k},\frac{-1}{2k+1}]\times [\frac{-1}{2k},\frac1{2k}]^{n-1}$, $k\in \mathbb N$.

We can now combine both embeddings $\Theta$~\eqref{e:theta} and $\Psi$~\eqref{e:psi} to define an embedding $\Phi: I^\omega\times I^\omega \to C(X)$ satisfying
\begin{multline}\label{Phireduction}
\Phi\bigl((q_k), (t_k)\bigr) \in   S(X)\cap CB(X) \quad\text{if and only if}\\
 \bigl((q_k), (t_k)\bigr) \in \partial(I^\infty)\times (0,1)^\infty,
  \end{multline}
by putting
\begin{equation}\label{Phi}
\Phi\bigl((q_k), (t_k)\bigr)=\Psi\bigl((q_k)\bigr)\cup \Theta\bigl((t_k)\bigr).
\end{equation}
Since $\partial(I^\infty)\times (0,1)^\infty$ is strongly $D_2(F_\sigma)$-universal in $I^\infty\times I^\infty$,  there exists, for each $D_2(F_\sigma)$-set $M\subset I^\infty$,  an embedding $\tau: I^\infty \to I^\infty\times I^\infty$ such that
$\tau^{-1}\bigl(\partial(I^\infty)\times (0,1)^\infty\bigr)=M$. Hence, the composition $f=\Phi\tau:I^\infty \to C(X)$  satisfies~\eqref{f}.
 Similarly as in~\eqref{e:varphi}, we can additionally assume that $(0,1)^n$ is an open subset of $X$ and
 \begin{equation}\label{eq:f}
 f(I^\infty)\subset C(I^n)\subset C(X).
 \end{equation}
\end{proof}

 \begin{theorem}\label{t2}
If  $X$ is $\pi$-$n$-Euclidean and no subset of dimension $\le 1$ separates $X$, then  there is a homeomorphism $h:I^\infty\times I^\infty \to C(X)$ such that $h\bigl(\partial I^\infty\times (0,1)^\infty\bigr) = S(X)\cap CB(X)$.
\end{theorem}

\begin{proof}
Since  subsets of dimension $\le 1$ do not separate $X$, $n=\dim X\ge 3$.

(1) By Proposition~\ref{p5}, family $S(X)\cap CB(X)$ is $D_2(F_\sigma)$.

(2) To prove the strong $D_2(F_\sigma)$-universality, we proceed similarly as in the proof of   Lemma~\ref{l5}. We just replace $\varphi$ with embedding $f$ from Proposition~\ref{p6} satisfying~\eqref{eq:f} and  family $C(X)\setminus CB(X)$ with $S(X)\cap CB(X)$.  Then, for an arbitrary fixed $D_2(F_\sigma)$-set $M\subset I^\infty$,  we have

 \emph{$\varphi_U(q)\in S(X)\cap CB(X)$ if and only if  $f(q)\in S(X)\cap CB(X)$ if and only if $q\in M$.}

For $q\notin K$, attaching finitely many pairwise disjoint sets $\varphi_U(q)$,  to the one-dimensional part $A(q)$ at single points does not change the above equivalences  (no one-dimensional set separates $X$) which means that the required embedding $g$ satisfies
$$g^{-1}\bigl(C(X)\setminus CB(X)\bigr)\setminus K=M\setminus K.$$

(3) $S(X)\cap CB(X)$ is contained in the $F_\sigma$-absorber $S(X)\cap C(X)$  in $C(X)$ (see~\cite[Theorem 5.3]{KS}), so in  a  $\sigma Z$-set in $C(X)$.

Thus we conclude that  $S(X)\cap CB(X)$ is a $D_2(F_\sigma)$-absorber in $C(X)$ and the required  homeomorphism $h$ exists.
\end{proof}

\section*{Acknowledgments}
The author is indebted to the organizers of the 11th Workshop on Continuum Theory and Hyperspaces at UNAM, Cuernavaca, M\'{e}xico, 2017, for invitation and its participants for inspiring discussions on the  class $CB(X)$.
 Rusell-Aaron Qui\~{n}ones-Estrella  initiated the study of continua with connected boundaries (at the 10th Workshop in 2016) and Norberto  Ordo\~{n}ez asked about the Borel complexity and characterization of  $CB(X)$; they are also coauthors of  recent paper~\cite{EQOV} devoted to the continua.

\bibliographystyle{amsplain}

\end{document}